\documentclass[11pt]{amsart}
\usepackage{amssymb,pb-diagram}
\usepackage{amsmath}
\usepackage{amsthm,amssymb,amsfonts}
\usepackage{amssymb,amscd}
\usepackage{epic,eepic,epsfig}
\usepackage{latexsym}
\usepackage{xy}
\xyoption{all}

\usepackage{amssymb,amsthm,enumerate,color,mathrsfs}

\def\Lip{\operatorname{Lip}}

\newtheorem{Theorem}{Theorem}[section]

\newtheorem{Proposition}[Theorem]{Proposition}
\newtheorem{Lemma}[Theorem]{Lemma}
\newtheorem{Corollary}[Theorem]{Corollary}

\theoremstyle{definition}

\theoremstyle{definition}

\theoremstyle{remark}

\def \cal{\mathcal}

\def\Lip{\text{Lip}}

\def\eps{\varepsilon}

\def\Ndb{\mathbb N}

\def\Rdb{\mathbb R}

\def\Zdb{\mathbb Z}
\def\eps{\varepsilon}

\def\Lip{\text{Lip}}

\newcommand{\R}{\ensuremath{\mathbb{R}}}

\newcommand{\bib}{\bibitem}

\begin{document}

\title{Approximation and Schur properties for Lipschitz free spaces over compact metric spaces}

\author{P. H\'{a}jek}
\address{Mathematical Institute, Czech Academy of Science, \v{Z}itn\'{a} 25, 115 67 Praha 1,
Czech Republic, and Department of Mathematics, Faculty of Electrical Engineering, Czech Technical University in Prague, Zikova 4, 160 00, Praha, Czech Republic.}
\email{hajek@math.cas.cz}

\author{G. Lancien}

\address{Universit\'{e} de Franche-Comt\'{e}, Laboratoire de Math\'{e}matiques UMR 6623,
16 route de Gray, 25030 Besan\c{c}on Cedex, France.}
\email{gilles.lancien@univ-fcomte.fr}

\author{E. Perneck\'{a}}
\address{Universit\'{e} de Franche-Comt\'{e}, Laboratoire de Math\'{e}matiques UMR 6623,
16 route de Gray, 25030 Besan\c{c}on Cedex, France and Institute of Mathematics, Polish Academy of Sciences, ul. \'{S}niadeckich 8, 00-956 Warszawa, Poland.}
\email{pernecka@karlin.mff.cuni.cz}



\thanks{The first author was supported by the grant GACR 201/11/0345}


\maketitle

\begin{abstract} We prove that for any separable Banach space $X$, there exists a compact metric space which is homeomorphic to the Cantor space and whose Lipschitz-free space contains a complemented subspace isomorphic to $X$. As a consequence we give an example of a compact metric space which is homeomorphic to the Cantor space and whose Lipschitz-free space fails the approximation property and we prove that there exists an uncountable family of topologically equivalent distances on the Cantor space whose free spaces are pairwise non isomorphic. We also prove that the free space over a countable compact metric space has the Schur property. These results answer questions by G. Godefroy.
\end{abstract}

\markboth{}{}

\section{Introduction}
Let $(M,d)$ be a metric space equipped with a distinguished point denoted $0$. We denote by $\Lip_0(M)$ the Banach space of all real-valued Lipschitz functions on $M$ that vanish at $0$ endowed with its natural norm:
$$\forall f\in \Lip_0(M),\ \ \|f\|_{\Lip}=\sup\left\{\frac{|f(x)-f(y)|}{d(x,y)},\ x,y\in M,\ x\neq y\right\}.$$
The Dirac map $\delta: M \to \text{Lip}_0(M)^*$ defined by $\langle f, \delta(x)\rangle =f(x)$ for $f\in \text{Lip}_0(M)$ and $x\in M$ is an isometric
embedding from $M$ into $\text{Lip}_0(M)^*$. The closed linear span of $\{\delta(x),\ x\in M\}$ in $\Lip_0(M)^*$ is denoted $\cal F (M)$ and called the \emph{Lipschitz-free space over} $M$ (free space in short). The space $\cal F (M)$ is an isometric predual of $\Lip_0(M)$. The fundamental property of the free spaces is that any Lipschitz function between metric spaces extends, via the $\delta$ maps, to a continuous linear map between the corresponding free spaces (see \cite{GK}, \cite{K} or \cite{W}). This feature makes it very natural to study the linear structure of the free spaces. This direction of research, which has become very active, was initiated in a paper by G. Godefroy and N. Kalton \cite{GK}, and soon after was continued in an article by N. Kalton \cite{K}. In \cite{GK}, the authors concentrated on the study of ${\cal F}(X)$, when $X$ is a Banach space. In that situation, there is another fundamental object, namely the \emph{barycenter map} $\beta$, which is defined to be the unique bounded
linear map from $\cal F(X)$ onto $X$ such that for all $x\in X$, $\beta\delta(x)=x$. Note that $\|\beta\|=1$. One of the main results in \cite{GK} is that if $X$ is a separable Banach space, then there exists a linear isometry $U: X \to \cal F(X)$ such that $\beta U=Id_X$. Let us point out that it readily follows that $P=U\beta$ is a non expansive projection from $\cal F(X)$ onto $U(X)$.\\
Let us now recall that a Banach space $X$ has the approximation property (AP in short) if for any compact subset $K$ of $X$ and any $\eps>0$, there exists a bounded finite rank operator $T$ on $X$ such that $\|T(x)-x\|\le \eps$ for all $x\in K$. For $\lambda\ge 1$, $X$ has the $\lambda$-bounded approximation property ($\lambda$-BAP) if for any compact subset $K$ of $X$ and any $\eps>0$, there exists a bounded finite rank operator $T$ on $X$ such that $\|T\|\le \lambda$ and $\|T(x)-x\|\le \eps$ for all $x\in K$. We say that $X$ has the bounded approximation property (BAP) if it has the $\lambda$-BAP for some $\lambda\ge 1$ and that it has the metric approximation property (MAP) if it has  the 1-BAP. Another important result of \cite{GK} is that a Banach space $X$ has the $\lambda$-BAP if and only if ${\cal F}(X)$ has the $\lambda$-BAP. It follows that the BAP is stable under Lipschitz isomorphisms. It is then natural to study these approximation properties for free spaces over other metric spaces and in
particular over compact metric spaces. Let us summarize what is known on this subject. G. Godefroy and N. Ozawa proved in \cite{GO} that there exists a compact convex subset $K$ of a Banach space $X$ such that $\cal F(K)$ fails the AP. It is proved in \cite{LP} that the free space over a doubling metric space has the BAP. In \cite{D1} A. Dalet showed that the free space over a countable compact metric space has the MAP (see also \cite{D2} for the extension of this result to the case of countable proper metric spaces). In \cite{PS}, R. Smith and the third author proved, in particular, that the free space of a compact convex subset of a finite dimensional normed space has the MAP. Finally, let us mention that M. C\'{u}th and M. Doucha recently obtained in \cite{CD} that the free space of a separable ultrametric space has the MAP.

In \cite{GO}, G. Godefroy and N. Ozawa also proved that if $K$ is a ``small Cantor set'', then $\cal F(K)$ has the MAP. Then G. Godefroy raised the question whether the free space over a perfect totally disconnected compact metric space (or equivalently a metric space homeomorphic to the Cantor space $C$) has the BAP \cite{G2} (this question was already mentioned in \cite{CD}). In section \ref{Cantor} we prove the existence of a compact metric space $K$, which is homeomorphic to $C$ and such that $\cal F(K)$ fails the AP. Our proof is based on an adaptation of the construction of a linear isometric lifting of the barycenter map $\beta$ when $X$ is a separable Banach space (see \cite{GK} for the original proof). We will follow the elementary approach given in \cite{G}. This construction will also allow us to build an uncountable family of compact metric spaces, all homeomorphic to the Cantor space, but whose free spaces are pairwise non isomorphic.

Let us finally recall that a Banach space $X$ is weakly sequentially complete if a sequence $(x_n)_{n=0}^\infty$ in $X$ is weakly converging whenever the sequences $(x^*(x_n))_{n=0}^\infty$ are converging in $\Rdb$ for all $x^*\in X^*$. Then $X$ has the Schur property if every weakly converging sequence in $X$ is norm converging. Clearly, a Banach space with the Schur property is weakly sequentially complete. Recently, M. C\'{u}th, M. Doucha and P. Wojtaszczyk proved that $\cal F(\Rdb^n)$ is weakly sequentially complete. Among other questions, G. Godefroy proposed to study the weak sequential completeness of free spaces over compact metric spaces and specifically asked whether the free space over a countable compact metric space is weakly sequentially complete. In section \ref{Schur}, we address this question and show that the free space of any countable compact (or even proper) metric space has the Schur property.

\section{Free spaces over Cantor sets}\label{Cantor}

We start by stating the main result of this section.

\begin{Theorem}\label{main} Let $X$ be a separable Banach space and $\eps>0$. Then there exists a compact subset $K$ of $X$ such that $K$ is homeomorphic to the Cantor space $C$, a linear embedding $R$ from $X$ into $\cal F (K)$ such that $\|R\|\,\|R^{-1}\|\le 1+\eps$ and a projection $P$ from $\cal F (K)$ onto $R(X)$ with $\|P\|\le 1+\eps$.
\end{Theorem}

As an immediate Corollary, we obtain:

\begin{Corollary} There exists a compact metric space $K$ homeomorphic to the Cantor space $C$ such that $\cal F(K)$ fails the approximation property.
\end{Corollary}

\begin{proof} Let $X$ be a separable Banach space failing the approximation property. The existence of such a space was proved by P. Enflo in \cite{E}. Then Theorem \ref{main} insures the existence of a compact subset $K$ of $X$ so that $K$ is homeomorphic to $C$ and $X$ is isomorphic to a complemented subspace of $\cal F(K)$. Therefore $\cal F(K)$ fails the approximation property.
\end{proof}

\begin{proof}[Proof of Theorem \ref{main}] The idea of the proof is to build a ``very fat spanning Cantor set'' in any separable Banach space. So let $X$ be a separable Banach space and $\eps>0$. We choose a small enough $\eta \in (0,1/2)$. The choice of $\eta$ will be made precise along our construction. It follows from the work of I. Singer \cite{S} (see also \cite{HMVZ} Corollary 1.26 page 12) that there exists $(x_n)_{n=1}^\infty$  in $X$ and $(x_n^*)_{n=1}^\infty$ in $X^*$ such that the linear span of $\{x_n,\ n\ge 1\}$ is dense in $X$ and so that for all $n\ge 1$, $\|x_n\|=2^{-n}=x_n^*(x_n)$, $\|x_n^*\|\le 1+\eta$ and for $n\neq m$, $x_m^*(x_n)=0$.\\
We now pick a sequence $(a_n)_{n=1}^\infty$ in $(0,1)$ such that $a_n\le \eta 2^{-{n}}$ for all $n\ge1$, and $\prod_{n=1}^\infty (1-a_n)=\Pi\ge 1-\eta$.\\
For all $n\ge 1$, we choose a subset $C_n$ of $[0,1]$ which is homeomorphic to the Cantor space $C$ and such that its Lebesgue measure satisfies $\lambda(C_n)= 1-a_n$. We also ask that $\{0,1\}\subset C_n$.\\
For $N\ge 1$ and $1\le n\le N$, we denote $A_N^n=\{k\neq n,\ 1\le k\le N\}$ and
$$D_N=\prod_{k=1}^N C_k,\ \ D_N^n=\prod_{k\in A_N^n}C_k,\ \
\Pi_N=\prod_{k=1}^N (1-a_k)\ \ {\rm and}\ \ \Pi_N^n=\prod_{k\in A_N^n} (1-a_k).$$
We also denote $\lambda_N$ the Lebesgue measure on $[0,1]^N$ and $\lambda_N^n$ the Lebesgue measure on $[0,1]^{A_N^n}$.\\
Then we define $f:[0,1]^\Ndb \to X$ by $f(t)=\sum_{n=1}^\infty t_nx_n$, for $t=(t_n)_{n=1}^\infty \in [0,1]^\Ndb$. Note that $f$ is one to one and continuous. It follows that $K=f(\prod_{n=1}^\infty C_n)$ is homeomorphic to $C$. We will show that, with the right choice of $\eta$, $K$ will provide us with the desired example.

For $N\ge 1$, we denote $E_N$ the linear span of $\{x_1,..,x_N\}$ and define $R_N$ to be the linear map from $E_N$ to $\cal F(K)$ such that
$$\forall n\le N \ \ R_N(x_n)=\frac{1}{\Pi_N^n}\int_{D_N^n} \big(\delta(x_n+\sum_{k\in A_N^n}t_kx_{k})-\delta(\sum_{k\in A_N^n}t_{k}x_{k})\big)\,d\lambda_N^n(t).$$
Note that for all $x\in E_N$, $\beta R_N(x)=x$ and for all $n\le N$, $\|R_N(x_n)\|\le \|x_n\|$. Then we consider $S_N:E_N\to \cal F(K)$ to be the linear map such that
$$\forall n\le N \ \ S_N(x_n)=\Pi_N^nR_N(x_n)$$
and $U_N:E_N\to \cal F(X)$ to be the linear map such that
$$\forall n\le N \ \ U_N(x_n)=\int_{[0,1]^{A_N^n}} \big(\delta(x_n+\sum_{k\in A_N^n}t_{k}x_{k})-\delta(\sum_{k\in A_N^n}t_{k}x_{k})\big)\,d\lambda_N^n(t).$$
By the choice of $(a_n)_{n=1}^\infty$,
\begin{equation}\label{eq1}\forall n\le N,\ \ \|\Pi_NR_N(x_n)-S_N(x_n)\|=|\Pi_N-\Pi_N^n|\,\|R_N(x_n)\|\le \eta 2^{-n}\|x_n\|.
\end{equation}
Consider now $x=\sum_{n=1}^N \alpha_nx_n \in E_{N}$. Note that for all $n\le N$, $x_n^*(x)=\alpha_n 2^{-n}$, which implies that $|\alpha_n|\le 2^n(1+\eta)\|x\|$. Then, we deduce from (\ref{eq1}) that
\begin{equation}\label{eq2}
\|\Pi_NR_N(x)-S_N(x)\|\le \sum_{n=1}^{{N}}|\alpha_n|\eta 2^{-n}\|x_n\|\le \eta(1+\eta)\|x\|\le 2\eta \|x\|.
\end{equation}
We have shown that $\|\Pi_NR_N-S_N\|\le 2\eta$.

Assume now that $f\in \text{Lip}_0(X)$ is such that its restriction to $E_N$, that we denote $f_N$, is continuously differentiable. Let $x=\sum_{n=1}^N \alpha_nx_n \in E_{N}$. Then
$$\langle f,S_N(x)\rangle -\int_{{D_N}} \langle \nabla f_N\big(\sum_{k=1}^N t_kx_k\big),x\rangle\,d\lambda_N(t)= $$
$$ \sum_{n=1}^N \alpha_n\int_{{D_N^n}}\Big(\int_{[0,1]\setminus C_n}\langle \nabla f_N\big(\sum_{{k}=1}^N t_kx_k\big),x_n\rangle\,{dt_n}\Big)\,{d\lambda_N^n(t)}.$$
It follows that
$$\big|\langle f,S_N(x)\rangle -\int_{{D_N}} \langle \nabla f_N\big(\sum_{k=1}^N t_kx_k\big),x\rangle\,d\lambda_N(t)\big| \le \sum_{n=1}^N a_n|\alpha_n|\,\|f\|_{\text{Lip}}\|x_n\|$$
$$\le (1+\eta)\|f\|_{\text{Lip}}\|x\|\sum_{n=1}^N a_n\le \eta(1+\eta)\|f\|_{\text{Lip}}\|x\| \le 2\eta\|f\|_{\text{Lip}}\|x\|.$$
We also have that
$$\big|\int_{[0,1]^N} \langle \nabla f_N\big(\sum_{k=1}^N t_kx_k\big),x\rangle\,d\lambda_N(t)-\int_{{D_N}} \langle \nabla f_N\big(\sum_{k=1}^N t_kx_k\big),x\rangle\,d\lambda_N(t)\big|$$
$$\le (1-\Pi_N)\|f\|_{\text{Lip}}\|x\|\le \eta \|f\|_{\text{Lip}}\|x\|.$$
Hence
\begin{equation}\label{eq3}
\big|\langle f,S_N(x)\rangle -\int_{[0,1]^N} \langle \nabla f_N\big(\sum_{k=1}^N t_kx_k\big),x\rangle\,d\lambda_N(t)\big| \le 3\eta\|f\|_{\text{Lip}}\|x\|.
\end{equation}
Then we can use a now classical approximation argument, using the convolution by a smooth approximation of unity in $E_N$, as it is presented in \cite{G}, to deduce from (\ref{eq3}) that $\|S_N-U_N\|\le 3\eta$ and therefore that $\|\Pi_NR_N-U_N\|\le 5\eta$. Since we know from \cite{GK} that $U_N$ is an isometry, we get that
$$\|R_N\|\le \frac{1+5\eta}{\Pi_N}\le \frac{1+5\eta}{1-\eta}\le 1+\eps,$$
if $\eta$ was initially chosen small enough.

The last step of the proof is to define $R$ as a ``limit'' of  the sequence $(R_N)_{N=1}^\infty$. Indeed, for all $n\le N$:
$$R_{N+1}(x_n)=\frac{1}{\Pi_{N+1}^n}\int_{D_{N+1}^n} \big(\delta(x_n+\sum_{k\in A_{N+1}^n}t_{k}x_{k})-\delta(\sum_{k\in A_{N+1}^n}t_{k}x_{k})\big)\,d\lambda_{N+1}^n(t)$$
and
$$R_N(x_n)=\frac{1}{\Pi_N^n}\int_{{D_N^n}} \big(\delta(x_n+\sum_{k\in A_N^n}t_{k}x_{k})-\delta(\sum_{k\in A_N^n}t_{k}x_{k})\big)\,d\lambda_N^n(t)$$
$$=\frac{1}{\Pi_{N+1}^n}\int_{D_{N+1}^n} \big(\delta(x_n+\sum_{k\in A_N^n}t_{k}x_{k})-\delta(\sum_{k\in A_N^n}t_{k}x_{k})\big)\,d\lambda_{N+1}^n(t).$$
We deduce that
$$\|R_{N+1}(x_n)-R_N(x_n)\|\le {2\|x_{N+1}\|\le 2^{-N}}.$$
It follows that for any $x\in E_N$, the sequence $(R_l(x))_{l=N}^\infty$ is norm convergent in $\cal F(K)$. We denote $R(x)$ its limit. Then $R$ extends to a bounded linear map $R:X\to \cal F(K)$ such that $\|R\|\le 1+\eps$ and for all $x\in X$, $\beta R(x)=x$. Since $\|\beta\|\le 1$, we obtain that $R$ is a linear embedding from $X$ into $\cal F(K)$ such that $$\forall x\in X,\ \ \|x\|\le \|R(x)\|\le (1+\eps)\|x\|.$$
Denote now $\beta_K$ the restriction of $\beta$ to $\cal F(K)$. It is easily checked that $P=R\beta_K$ is a bounded projection from $\cal F(K)$ onto $R(X)$, with $\|P\|\le 1+\eps$. This concludes our proof.

\end{proof}

Let us denote $\omega_1$ the first uncountable ordinal. As another application of Theorem \ref{main} we obtain the following result.

\begin{Corollary} There exists a family $(X_\alpha)_{\alpha <\omega_1}$ of separable Banach spaces and a family $(K_\alpha)_{\alpha <\omega_1}$ such that

\smallskip
(i) For all $\alpha<\omega_1$, $K_\alpha$ is a compact subset of $X_\alpha$ which is homeomorphic to the Cantor space $C$.

\smallskip
(ii) For all $\alpha<\omega_1$, $X_\alpha$ is isomorphic to a complemented subspace of $\cal F(K_\alpha)$.

\smallskip
(iii) For all $\beta<\alpha<\omega_1$, $X_\alpha$ is not isomorphic to a complemented subspace  of $\cal F({K}_\beta)$.

\smallskip\noindent In particular, for any $\alpha\neq \beta \in [1,\omega_1)$, $\cal F({K}_\alpha)$ and $\cal F({K}_\beta)$ are non isomorphic.
\end{Corollary}

\begin{proof} This construction will be done by transfinite induction.
Let $X_1$ be any separable Banach space (set $X_1=c_0$ for instance). By Theorem \ref{main}, there exists a compact subset $K_1$ of $X_1$ which is homeomorphic to $C$ and such that $X_1$ is isomorphic to a complemented subspace of $\cal F(K_1)$. We now use the work of Johnson and Szankowski \cite{JS}, who showed that there is no separable Banach space that is complementably universal (see also Theorem 2.11 in \cite{HMVZ}). Therefore, there exists a separable Banach space $X_2$ such that $X_2$ is not isomorphic to a complemented subspace of $\cal F(K_1)$. Then, we use Theorem \ref{main} to find {a} compact subset $K_2$ of $X_2$ homeomorphic to $C$ and such that $X_2$ is  isomorphic to a complemented subspace of $\cal F(K_2)$.\\
Assume now that $(X_\gamma)_{\gamma <\alpha}$ and $(K_\gamma)_{\gamma <\alpha}$ have been constructed for some $\alpha<\omega_1$ and satisfy properties (i), (ii) and (iii). \\
If $\alpha=\beta+1$ is a successor ordinal{,} Johnson and Szankowski's theorem insures the existence of a separable Banach space $Y$ such that $Y$ is not isomorphic to a complemented subspace of $\cal F({K}_\beta)$. Then we set $X_{\beta+1}=X_\beta \oplus Y$. So $X_{\beta+1}$ is not isomorphic to a complemented subspace of $\cal F(K_\beta)$. It also follows from our induction hypothesis that for all $\gamma<\beta$, $X_{\beta+1}$ is not isomorphic to a complemented subspace of $\cal F(K_\gamma)$. Then, by Theorem \ref{main}, there exists a compact subset $K_{\beta+1}$ of $X_{\beta+1}$ homeomorphic to $C$ such that $X_{\beta+1}$ is  isomorphic to a complemented subspace of $\cal F(K_{\beta+1})$.\\
Finally, if $\alpha$ is a limit ordinal, we set $X_\alpha=(\sum_{\gamma<\alpha} X_\gamma)_{\ell_1}$. Let $\beta<\alpha$, then $X_\alpha$ contains a complemented copy of $X_\gamma$, for all  $\gamma\in (\beta,\alpha)$. Therefore it follows from our induction hypothesis that for all $\beta<\alpha$, $X_\alpha$ is not isomorphic to a complemented subspace of $\cal F(K_\beta)$. Finally, we use again Theorem \ref{main} to produce a compact subset $K_{\alpha}$ of $X_\alpha$ such that $X_{\alpha}$ is  isomorphic to a complemented subspace of $\cal F(K_{\alpha})$. This concludes the transfinite induction.
\end{proof}

\noindent{\bf Remark.} Using this idea and the isometric lifting property of separable Banach spaces established in the original paper by Godefroy and Kalton \cite{GK}, one can build a family of separable Banach spaces $(Y_\alpha)_{\alpha<\omega_1}$ such that the spaces $\cal F(Y_\alpha)$ are pairwise non isomorphic. As well, using the construction of Godefroy and Ozawa \cite{GO}, one can build a family of compact convex subsets of separable Banach spaces $(C_\alpha)_{\alpha<\omega_1}$ such that the spaces $\cal F(C_\alpha)$ are pairwise non isomorphic.

\section{Schur property and free spaces over countable compact spaces}\label{Schur}

In \cite{GO} Godefroy and Ozawa showed that for any separable Banach space $X$ there is a compact convex subset $K$ of $X$ such that $X$ is isometric to a complemented subspace of $\cal F(K)$. Since weak sequential completeness passes to subspaces , this implies the existence of a compact convex subset of $c_0$ whose free space is not weakly sequentially complete. Similarly, Theorem \ref{main} yields the existence of a compact metric space homeomorphic to the Cantor space whose free space is not weakly sequentially complete. We will show, that like for the approximation property, the situation is completely different for countable compact spaces.

\begin{Theorem}\label{FSchur} Let $K$ be a countable compact metric space. Then $\cal F(K)$ has the Schur property.
\end{Theorem}

Before we proceed with the proof, we need to recall a very useful decomposition due to N. Kalton (Lemmas 4.1 and 4.2 in \cite{K}). If $M$ is a pointed metric space and $k<l\in \Zdb$, let $M_k=\{x\in M,\ d(x,0)\le 2^k\}$, $A_{k,l}=\{0\}\cup\{x\in M, 2^{k}\le d(x,0)\le 2^l\}$ and $C_k=A_{k-1,k+1}$. Then Kalton proved the following two results.

\begin{Lemma}\label{kalton1} Suppose $r_1<s_1<..<r_n<s_n$ belong to $\Zdb$. Then for any $\gamma_1,..,\gamma_n$ such that for all $k\le n$, $\gamma_k \in \cal F(A_{r_k,s_k})$ we have that
$$\big\|\sum_{k=1}^n\gamma_k\big\|_{\cal F(M)}\ge \frac13\sum_{k=1}^n\|\gamma_k\|_{\cal F(M)}.$$
\end{Lemma}

\begin{Proposition}\label{kalton2} For any $k\in \Zdb$ there exists an operator $T_k:\cal F(M)\to \cal F(C_k) \subset \cal F(M)$ such that
$$\forall \gamma \in \cal F(M)\ \ \gamma=\sum_{k\in \Zdb}T_k\gamma\ \ {\rm with}\ \  \sum_{k\in \Zdb}\|T_k\gamma\|_{\cal F(M)}\le \lambda \|\gamma\|_{\cal F(M)},$$
where $\lambda\ge1$ is a universal constant.
\end{Proposition}

\begin{proof}[Proof of Theorem \ref{FSchur}] We will prove by transfinite induction on $\alpha<\omega_1$, the following statement:\\
$(H_\alpha)$:
If $K^{(\alpha)}$ is finite, then $\cal F(K)$ has the Schur property (where $K^{(\alpha)}$ is the Cantor-Bendixon derived set of $K$ of order $\alpha$).

If $\alpha=0$, then $\cal F(K)$ is finite dimensional, so $(H_0)$ is clearly true.

Assume that $\alpha \in [1,\omega_1)$ and that $(H_\beta)$ is true for any $\beta <\alpha$. Consider now $K$ a compact metric space such that $K^{(\alpha)}=\{x_1,..,x_n\}$. It is then classical that there exist mutually disjoint clopen subsets $K_1$,..,$K_n$ of $K$ such that $K=\cup_{i=1}^n K_i$ and for all $i\le n$, $K_i^{(\alpha)}=\{x_i\}$. By compactness the $K_i$'s are uniformly far apart from each other. It then follows that $\cal F(K)$ is isomorphic to the $\ell_1$ sum of the $\cal F(K_i)$'s (see the proof of Theorem 3.1 in \cite{D1} for details). Therefore we may assume that $K^{(\alpha)}=\{0\}$, where $0$ is the origin of our pointed metric space.\\
Following the notation introduced for Lemmas \ref{kalton1} and \ref{kalton2}, we have that for any $k$ in $\Zdb$, $0$ is an isolated point of $C_k$. Therefore $C_k$ is a compact space such that $C_k^{(\alpha)}=\emptyset$. It follows from our induction hypothesis that $\cal F(C_k)$ has the Schur property. Consider now $(\gamma_i)_{i=1}^\infty$ a weakly null sequence in $\cal F(M)$. Then, we have that for any $k\in \Zdb$, the sequence $(T_k\gamma_i)_{i=1}^\infty$ is norm converging to $0$.\\
The rest of the proof goes by contradiction. So assume that $\|\gamma_i\|=\|\gamma_i\|_{\cal F(K)}$ does not tend to $0$. Then we will build a subsequence of $(\gamma_i)$ which is equivalent to the canonical basis of $\ell_1$, which is not weakly null. This will yield our contradiction. So, by extracting a first subsequence, we may assume that $\|\gamma_i\|\ge \delta$, for some $\delta >0$. Assume also, as we may, that the diameter of $K$ is less than 1. Fix $M>1$, whose (large) value will be made precise later.\\
Set $i_1=1$, $M_1=0$ and pick $N_1<0$ such that
$$\sum_{k<N_1} \|T_k\gamma_{i_1}\|<\frac{\delta}{M}.$$
We now choose inductively, using Proposition \ref{kalton2} and the fact that for all $k$ in $\Zdb$ $\lim_{i\to \infty}\|T_k\gamma_i\|=0$: $M_1=0>N_1>..>M_n>N_n$ and $i_1<..<i_n$ such that
$$\forall n\ge 0\ \ M_{n+1}=N_n-3\ \ {\rm and}\ \ \sum_{k\notin F_n} \|T_k\gamma_{i_n}\|<\frac{\delta}{M},\ {\rm where}\ \ F_n=[N_n,M_n].$$
Let now $(\alpha_n)$ be a finite sequence in $\Rdb$. We have that
\begin{equation}\label{eq1bis}
\big\|\sum_n\alpha_n\gamma_{i_n}\big\|=\big\|\sum_n\alpha_n(\sum_{k\in \Zdb}T_k\gamma_{i_n})\big\|\ge \big\|\sum_n\alpha_n(\sum_{k\in F_n}T_k\gamma_{i_n})\big\|-\frac{\delta}{M}\sum_n|\alpha_n|
\end{equation}
Since $\sum_{k\in F_n}T_k\gamma_{i_n} \in \cal F(A_{N_n-1,M_n+1})$ and $M_{n+1}=N_n-3$, it follows from Lemma \ref{kalton1} that
\begin{equation}\label{eq2bis}
\big\|\sum_n\alpha_n(\sum_{k\in F_n}T_k\gamma_{i_n})\big\|\ge \frac13\sum_n|\alpha_n|\,\big\|\sum_{k\in F_n}T_k\gamma_{i_n}\big\|\ge \frac13\big(\delta-\frac{\delta}{M}\big)\sum_n|\alpha_n|
\end{equation}
Combining inequalities (\ref{eq1bis}) and (\ref{eq2bis}), we get
$$\big\|\sum_n\alpha_n\gamma_{i_n}\big\|\ge \Big(\frac13\big(\delta-\frac{\delta}{M}\big)-\frac{\delta}{M}\Big)\sum_n|\alpha_n|\ge \frac{\delta}{6}\sum_n|\alpha_n|,$$
if $M$ was initially chosen large enough, which concludes our proof.

\end{proof}

\noindent{\bf Remark.} Gathering the results from \cite{D1} and Theorem \ref{FSchur}, we have that the free space over a countable compact metric space is a Schur space which is isometric to a dual space and has the metric approximation property. However it can still be quite complicated. Indeed, in \cite{CDW} an example of a compact subset $K$ of $\Rdb^2$ such that $K'=\{0\}$ and $\cal F(K)$ does not embed into $L^1$ is given.

\medskip
Recall that a proper metric space is a space whose balls are relatively compact. Then we can extend our result as follows.

\begin{Corollary} Let $M$ be a proper countable metric space. Then $\cal F(M)$ has the Schur property.
\end{Corollary}

\begin{proof} Still using the notation introduced for Lemmas \ref{kalton1} and \ref{kalton2}, we have that for any $k\in \Zdb$, $M_k$ is a countable compact metric space. Thus Theorem \ref{FSchur} implies that $\cal F(M_k)$ has the Schur property. We can now use Proposition 4.3 in \cite{K} which insures that $\cal F(M)$ is isomorphic to a subspace of $(\sum_{k\in \Zdb} \cal F(M_k))_{\ell_1}$. This clearly implies that $\cal F (M)$ has the Schur property.\\
Note that we could also have made the following reasoning. Since the $C_k$'s are countable and compact, each $\cal F(C_k)$ is a Schur space. Thus we can reproduce the argument used in our inductive proof of Theorem \ref{FSchur} to show that $\cal F (M)$ is a Schur space.

\end{proof}

\medskip\noindent {\bf Acknowledgements.} We would like to thank Gilles Godefroy for introducing us to these questions and for many insightful discussions on the subject of Lipschitz-free spaces. We also thank Antonin Proch\'{a}zka for many useful conversations.

\end{document}